\documentclass[12pt]{amsart}
\usepackage[OT4]{fontenc}
\usepackage{graphics,graphicx}
\usepackage{amsmath}
\usepackage{amssymb,amsfonts}
\usepackage {amsthm}
\usepackage{xcolor}

\usepackage{enumerate}
\usepackage{hyperref}
\usepackage[T2A]{fontenc}
\usepackage[cp1251]{inputenc}
\usepackage[all]{xy}

\newcommand\myfootnote[1]{
\renewcommand{\thefootnote}{}
\footnotetext{#1}
\def\thefootnote{\@arabic\c@footnote}
}

\makeatletter
\renewcommand{\subsection}{\@startsection{subsection}{2}{0mm}{-\baselineskip}{-5pt}{\it \bf}}
\makeatother

\newtheorem{theorem}{Theorem}
\newtheorem{lemma}{Lemma}

\newtheorem{example}{Example}
\newtheorem{remark}{Remark}

\title{DERIVATIONS AND AUTOMORPHISMS OF LOCALLY MATRIX ALGEBRAS}
\author{OKSANA BEZUSHCHAK}

\begin{document}

\maketitle

\address{Faculty of Mechanics and Mathematics, Taras Shevchenko National University of Kyiv, 
 60	Volodymyrska Street, Kyiv 01033, Ukraine}

\email{bezusch@univ.kiev.ua}

\begin{abstract}
We describe derivations and automorphisms of infinite tensor products of matrix algebras. Using this description we show that for a countable--dimensional locally matrix algebra $A$ over a field $\mathbb{F}$ the dimension of the Lie algebra of outer derivations of $A$ and the order of the group of outer automorphisms of $A$ are both  equal to $|\mathbb{F}|^{\aleph_0},$ where  $|\mathbb{F}|$ is the cardinality of the field $\mathbb{F}.$
	\end{abstract}



\section{Introduction and Main Results}

We study derivations and automorphisms of countable--dimensional locally matrix algebras.

Let $\mathbb{F}$ be a ground  field. Following \cite{Kurosh}, we call an associative {$\mathbb{F}$--algebra} $A$ a {\it locally matrix algebra} if for each finite subset of $A$ there exists a subalgebra $B\subset A$  containing this subset such that  $B\cong M_n(\mathbb{F})$ for some $n.$ We call a  locally matrix algebra $A$ {\it  unital} if it contains a unit $1.$

 J.~G.~Glimm \cite{Glimm} proved that every countable--dimensional unital locally matrix algebra is uniquely determined by its  Steinitz  number. In \cite{BezOl,BezOl_2}, we showed that this is no longer true for unital locally matrix algebras of uncountable dimensions.

 S.~A.~Ayupov and K.~K.~Kudaybergenov \cite{AyKud}   constructed an outer derivation of the countable--dimensional unital locally matrix algebra of Steinitz number $2^{\infty}$  and used it as an example of an outer derivation in a von Neumann regular simple algebra. In \cite{Strade}, H.~Strade studied derivations of locally finite--dimensional locally simple Lie algebras over a field of characteristic $0.$

 Recall that a linear map $d:A\rightarrow A$ is called  a  {\it derivation} if $$d(xy) \ = \ d(x)y \ + \  x d(y)$$ for arbitrary elements $x,$ $y$ from $A.$

 For an element $a\in A$ the adjoint operator $$\text{ad}_A(a): \ A \ \rightarrow \ A,  \ \ \ \ x \ \mapsto \ [a,x],$$ is an {\it inner derivation} of the algebra $A.$

Let   $\text{Der} (A)$ be the  Lie algebra of all derivations of the algebra $A$ and let $\text{Inder}(A) $ be the ideal of all inner derivations. The factor algebra    $$\text{Outder} (A) \ = \  \text{Der}(A) \, \diagup \, \raisebox{-2pt}{\text{Inder}(A)} $$ is called  the algebra of {\it outer derivations} of  $A.$

Let  $\text{Aut} (A)$ and  $\text{Inn} (A)$ be  the group of automorphisms and the group of inner automorphisms of the algebra $A,$ respectively. The factor group $$\text{Out} (A) \ = \  \text{Aut} (A) \, \diagup \, \raisebox{-2pt}{\text{Inn}(A)}$$ is called the group of {\it outer automorphisms} of  $A.$

Along with automorphisms of the algebra $A$ we consider the semigroup $P(A)$ of injective endomorphisms (embeddings) of $A,$ \ $\text{Aut} (A)  \subseteq P(A).$

In Sec.~2, we consider the Tykhonoff topology on the set $\text{Map} (A,A)$ of all mappings $A \rightarrow A$ and  prove the following Theorem.

\begin{theorem} \label{T1_new} Let $A$ be a locally matrix algebra.
  \begin{enumerate}
 		\item[$(1)$] \label{T1_1} 	The ideal    $\emph{Inder}(A)$  is dense in $\emph{Der} (A)$  in the Tykhonoff topology.
 		\item[$(2)$]  Let the algebra $A$  contains $1.$ Then the completion of  $\emph{Inn} (A)$ in $\emph{Map} (A,A)$ in the   Tykhonoff topology is the semigroup $P(A).$ In particular,
 $\emph{Inn} (A)$ is dense in  $\emph{Aut} (A).$
 		 	\end{enumerate}
 \end{theorem}

In Sec.~3, we describe derivations of an infinite tensor product of matrix algebras.

Let $I$ be an infinite set and let $\mathcal{P}$ be a system of nonempty finite subsets of $I.$ We say that the system  $\mathcal{P}$ is \textit{sparse} if
\begin{enumerate}
  \item[$(1)$] for any $S\in \mathcal{P}$ all nonempty subsets of $S$ also lie in $\mathcal{P},$
  \item[$(2)$] an arbitrary element $i\in I$ lies in no more than finitely many subsets from $\mathcal{P}.$
\end{enumerate}

Let \begin{equation}\label{otimes_eq} \mathbf{A} \ = \ \otimes_{i\in I} \ A_i \end{equation}
and  all  algebras $A_i$ are finite--dimensional matrix algebras over  $\mathbb{F}$. For a subset $S=\{i_1,\ldots, i_r \}\subset I$  the subalgebra $$A_S : \ = \ A_{i_1} \ \otimes \  \cdots \ \otimes \ A_{i_r}$$  is a tensor factor of the algebra $\mathbf{A}.$

Let $\mathcal{P}$ be a  system of nonempty finite subsets of $I.$ Let $f_S,$ $S\in \mathcal{P} ,$ be a system of linear operators $\mathbf{A}\rightarrow \mathbf{A}.$ The sum
\begin{equation}\label{sum_eq1}
\sum_{S\in \mathcal{P}} \ f_S
\end{equation}
converges in the   Tykhonoff topology if for an arbitrary element $a\in \mathbf{A}$ the set $$\big\{ \ S \ \in \ \mathcal{P} \ | \ f_S(a) \ \neq \  0 \ \big\}$$ is finite. In this case, the operator $$a \ \mapsto \ \sum_{S\in \mathcal{P}} \ f_S(a)$$ is a linear operator. Moreover, if every summand $f_S$ is a derivation of the algebra $\mathbf{A}$ then the sum   (\ref{sum_eq1}) is also a derivation of the algebra $\mathbf{A}.$

Let $\mathcal{P}$ be a sparse system. For each subset $S\in \mathcal{P}$ choose an element $a_S \in A_S.$ The  sum
\begin{equation}\label{eq4}
\sum_{S\in \mathcal{P}} \text{ad}_{\mathbf{A}} (a_S)
\end{equation}
converges in the  Tykhonoff topology to a derivation of $\mathbf{A}.$ Indeed, choose an arbitrary element $a\in \mathbf{A}.$ Let $$a \ \in \ A_{i_1} \ \otimes \ \cdots \ \otimes \ A_{i_r}.$$ Because of the sparsity of the  system $\mathcal{P},$ for all but finitely many subsets $S\in \mathcal{P}$  we have $$\big\{i_1, \ \ldots, \ i_r\big\} \ \cap \ S \ = \ \emptyset, \ \ \ \ \text{and therefore} \ \ \ \ \text{ad}_{\mathbf{A}}(a_S)(a)\ = \ 0.$$ Let $D_{\mathcal{P}}$  be the vector space of all  sums (\ref{eq4}),  $D_{\mathcal{P}}\subseteq \text{Der} (\mathbf{A}).$

For each algebra $A_i,$  $i\in I,$ choose a subspace $A_i^0$ such that
\begin{equation}\label{eq5}
  A_i \ = \  \mathbb{F} \ \cdot \ 1_{A_i} \ \dot{+} \  A_i^0
\end{equation}
is a direct sum,  $1_{A_i}$ is a unit element of $A_i.$ Let $E_i$ be a basis of $A_i^0.$ For a subset $S= \{i_1,\ldots, i_r\}$ of the set $I$ let
$$ E_S: \ = \ E_{i_1} \ \otimes \ \cdots \ \otimes \ E_{i_r} \ = \ \{\ a_{1}  \otimes  \cdots  \otimes a_{r} \ | \ a_k \in  E_{i_k}, \ 1 \leq  k  \leq  r\ \} .$$
and $$\text{ad}_{\mathbf{A}}(E_S)\ =\ \{ \ \text{ad}_{\mathbf{A}}(e) \ | \ e \ \in \ E_S \ \}.$$

A description of derivations of the algebra (\ref{otimes_eq}) is given by the following Theorem.

\begin{theorem} \label{T2_new}
\begin{enumerate}
  \item[$(1)$] Suppose that the set $I$ is countable.  Then $$\emph{Der} (\mathbf{A}) \ = \ \bigcup_{\mathcal{P}} \ D_{\mathcal{P}},$$ where the union is taken over all sparse systems of subsets of $I.$
  \item[$(2)$] \label{L1} Let $I$ be an infinite (not necessarily countable) set. Let $\mathcal{P}$ be a sparse system of subsets of $I.$ Then the union of finite sets of operators  $$ \bigcup_{S\in\mathcal{P}} \ \emph{ad}_{\mathbf{A}}(E_S) $$  is a topological basis of $D_{\mathcal{P}}.$
\end{enumerate}
  \end{theorem}

In Sec.~4, we prove the analog of the result of  H.~Strade \cite{Strade} for locally finite--dimensional locally simple Lie algebras.

\begin{theorem} \label{T3} Let $A$ be a countable--dimensional  locally matrix algebra. Then the Lie algebra  $\emph{Outder} (A)$ is not locally finite--dimensional.  \end{theorem}

In Sec.~5, we describe automorphisms and unital injective endomorphisms of a  countable--dimensional unital locally matrix algebra $A.$ Remark, that by the result of A.~G.~Kurosh (\cite{Kurosh},  Theorem 10) the semigroup $P(A)$ of  unital injective homomorphisms is strictly bigger than $\text{Aut} (A).$

The starting point here is Koethe's Theorem   \cite{Koethe} stating that every countable--dimensional unital locally matrix algebra $A$ is isomorphic to a countable tensor product of matrix algebras. Therefore
\begin{equation}\label{T_PR0}
  A\cong \otimes_{i=1}^{\infty} A_i, \quad A_i \cong M_{n_i} (\mathbb{F}), \quad i\geq 1.
\end{equation}

Let $H_n$ be the subgroup of the group $\text{Inn} (A)$ generated by conjugations by invertible elements from $$\otimes_{i\geq n} \ A_i.$$ Clearly, $$H_n \ \cong \ \text{Inn} \ \big(\otimes_{i\geq n} A_i \ \big) $$  and $$\text{Inn}(A) \ = \  H_1 \ >  \ H_2 \ > \ \cdots .$$ For each $n\geq 1$ choose a system of representatives of left cosets $hH_{n+1},$ $h\in H_n,$ and denote it as $\mathcal{X}_n.$ We assume that each $\mathcal{X}_n$ contains the identical automorphism.

For an arbitrary sequence of automorphisms $\varphi_n\in \mathcal{X}_n,$ $n\geq 1,$ the infinite product $\varphi=\varphi_1\varphi_2 \cdots$ converges in the Tykhonoff topology. Clearly, $\varphi\in P(A).$

\begin{theorem} \label{T4_new} An arbitrary unital injective endomorphism $\varphi \in P(A)$ can be uniquely represented as
\begin{equation*}\label{varphi}
\varphi \ = \ \varphi_1 \ \varphi_2 \ \cdots,
\end{equation*}
where $\varphi_i\in \mathcal{X}_i$ for each $i\geq 1.$  \end{theorem}

We call a sequence of automorphisms $\varphi_n\in H_n,$ $n\geq 1,$ \emph{integrable} if for an arbitrary element $a\in A$ the subspace spanned by all elements $$ \varphi_n \ \varphi_{n-1} \ \cdots \ \varphi_1(a), \ \ \  \ n \geq 1 ,$$ is finite--dimensional.

\begin{theorem} \label{T5_new} An  injective endomorphism  $$\varphi \ = \ \varphi_1 \ \varphi_2 \ \cdots, \ \ \ \ \text{where} \ \ \ \  \varphi_n\in H_n,  \ \ \ \ n\geq 1,$$ is an automorphism if and only if the sequence
\begin{equation}\label{EQ_integrable}
 \big\{ \ \varphi_n^{-1} \ \big\}_{n\geq 1}
\end{equation}
is integrable. \end{theorem}

\begin{example}\label{example1} In each $A_i,$ $i\geq 1,$ choose an invertible element $a_i.$ Let $\hat{a_i}$ be the conjugation automorphism by $a_i.$ Then the sequence $$\big\{ \ \hat{a_i}^{-1} \ \big\}_{i\geq 1}$$ is integrable.
\end{example}

\begin{example}\label{example2} Let $e_{pq}$ denote a matrix unit. Let  $A_i \cong M_{n_i}(\mathbb{F}),$ $i\geq 1,$ and assume that $1\leq p,\ q\leq n_i$ so that $e_{pq}$ can be thought of as a matrix unit of $M_{n_i}(\mathbb{F}).$ Denote $$e_{pq}(i) \ = \ \underbrace{1\otimes \cdots \otimes 1 \otimes e_{pq}}_i  \otimes  1  \otimes  \cdots \ \in \ A_i \ \subset \ A.$$  Let $$a_i \ = \ e_{11}(i) \ e_{12}(i+1).$$ Clearly, $a_i^2=0.$ Let $\phi_i$ denote the conjugation by the element $(1+ a_i)^{-1}.$ The sequence
\begin{equation}\label{EQ_4-0}
\big\{ \ \phi_i^{-1} \ \big\}_{i\geq 1}
\end{equation}
is not integrable. Hence, $\phi = \phi_1 \phi_2 \cdots$ is an injective endomorphism that is not an  automorphism. \end{example}

\begin{remark} This example provides another proof of   Theorem $10$ \emph{\cite{Kurosh}} of {A.~G.~Kurosh.} \end{remark}

In Sec.~6, we determine dimensions of Lie algebras $\text{Der} (A)$ and  \\ $\text{Outder} (A)$ and orders of groups $\text{Aut} (A)$ and $\text{Out} (A),$ where $A$ is a countable--dimen\-sio\-nal   locally matrix algebra.

We denote the cardinality of a set $X$  as $|X|.$ For  two sets $X$ and $Y$ let  $\text{Map}(Y,X)$ denote the set of all mappings from $Y$ to $X.$ Given two cardinals $\alpha,$ $\beta$ and sets $X,$ $Y$ such that $|X|=\alpha,$ $|Y|=\beta$ we define  $\alpha^{\beta} = |\text{Map}(Y,X)|.$ As always $\aleph_0$ stands for the countable cardinality.

\begin{theorem}\label{T6_new_P1}
Let $\mathbf{A}= \otimes_{i\in I} A_i,$ where $I$ is an infinite  set and each algebra $A_i$ is a  matrix algebra over a field $\mathbb{F}$ of the dimension $>1.$ Then
\begin{equation}\label{eq1}
\dim_{\mathbb{F}} \emph{Der} (\mathbf{A}) \ = \  \dim_{\mathbb{F}} \emph{Outder} (\mathbf{A}) \ = \ | \, \mathbb{F} \, |^{|I|}.
\end{equation}
\end{theorem}

 \begin{theorem} \label{T7_new_T2} 
  Let $A$ be a countable--dimensional  locally matrix algebra over a field $\mathbb{F}$. Then
 \begin{equation}\label{eq2}
\dim_{\mathbb{F}} \emph{Der} (A) \ = \  \dim_{\mathbb{F}} \emph{Outder} (A) \ = \ | \, \mathbb{F} \, |^{\aleph_0}.
\end{equation}  \end{theorem}

\begin{remark} For many uncountable cardinals $|\mathbb{F}|$ we have $|\mathbb{F}|^{\aleph_0} =|\mathbb{F}|.$ For example, this is the case when $|\mathbb{F}|=\lambda^{\mu}$ is a power of cardinals and $\mu \geq \aleph_0 .$ If $\mathbb{F}=\mathbb{F}_0 (z,z^{-1})$ is the field  of Laurent series over some field $\mathbb{F}_0,$ or its algebraic extension, then $|\mathbb{F}|=|\mathbb{F}_0|^{\aleph_0}, $ and therefore $|\mathbb{F}|^{\aleph_0}=|\mathbb{F}|. $
 \end{remark}

  \begin{remark} A locally matrix   algebra $A$ over a field of zero characteristic gives rise to the locally finite--dimensional  locally simple Lie algebra $L=[A,A].$ Moreover, there are embeddings $$\emph{Der} (A) \ \rightarrow \ \emph{Der} (L),\ \ \ \ \emph{Inder} (A) \ \rightarrow \ \emph{Inder} (L), $$ $$ \emph{Outder} (A) \ \rightarrow \ \emph{Outder} (L),$$ which imply  $$\dim_{\mathbb{F}} \emph{Outder} (L) \ \geq \ \dim_{\mathbb{F}} \emph{Outder} (A).$$  Therefore, Theorem \emph{\ref{T7_new_T2}} contradicts Theorem $3.2$ from \emph{\cite{Strade}}. \end{remark}

\begin{theorem}\label{T8_new_P1}
Let $A$ be a countable--dimensional  locally matrix algebra over a field $\mathbb{F}.$ Then
\begin{equation}
|\, \emph{Aut} (A)\, | \ = \ | \, \emph{Out}(A)\, |  \ = \ | \, \mathbb{F}|^{\aleph_0} .
\end{equation}
\end{theorem}

\section{Tykhonoff topology}

Let $X$ be an arbitrary set. The set $\text{Map} (X,X)$ of mappings {$X \rightarrow X$}  is equipped with the Tykhonoff topology (see \cite{Willard}). For distinct elements $a_1,\ldots,a_n\in X$ and arbitrary elements $b_1,\ldots,b_n\in X,$ $n\geq 1,$ consider the subset
\begin{equation}\label{eq3}
 M(a_1,\ldots,a_n;b_1,\ldots,b_n) = \big\{ f:X\rightarrow X \ | \ f(a_i)=b_i, \ 1\leq i\leq n \big\}
\end{equation}
 of  $\text{Map} (X,X)$. The Tykhonoff topology on $\text{Map} (X,X)$ is generated by all open sets of this type (\ref{eq3}).

Thus, for a subset $S \subset \text{Map} (X,X)$ a mapping $f:X\rightarrow X$ lies in the completion $\overline{S}$ of $S$ if and only if for  any $n \geq 1$ and for any  elements $a_1,\ldots,a_n\in X$ there exists a mapping $$g:X\rightarrow X, \ \ g\in S, \ \text{ such that } \ f(a_i)=g(a_i), \ \ 1\leq i \leq n.$$

\begin{proof}[Proof of Theorem $\ref{T1_new}$] (1) It is  straightforward that the set of derivations \ $\text{Der} (A)$ is closed in $\text{Map} (X,X).$ It implies that  $$\overline{\text{Inder} (A)}\  \subseteq  \ \text{Der} (A).$$ To prove the assertion we need to show that for any derivation $d: A\rightarrow A$ and arbitrary elements $a_1,\ldots,a_n \in A$ there exists an element $b\in A$ such that $$[b  ,  a_i] \ = \ d(a_i), \quad 1\leq i \leq n.$$

Choose a subalgebra $B_1 \subset A$ such that $a_1,\ldots, a_n \in B_1$ and $B_1 \cong M_k(\mathbb{F}).$  Then choose a subalgebra $B_2 \subset A$ such that $$B_1+ d(B_1)\subseteq B_2 \quad \text{and} \quad B_2 \cong M_l(\mathbb{F}). $$  The vector space $B_2$ is a $B_1$--bimodule and  $d:B_1\rightarrow B_2 $ is a bimodule derivation. Since any bimodule derivation of finite--dimensional matrix algebras over a field is inner (see \cite{Pierce}) there exists an element $b\in B_2$ such that $d(a)=[b,a]$ for all elements $a\in B_1.$ This proves the part (1) of the Theorem.

\vspace{8pt}
(2) Let $A$ be a unital locally matrix algebra. The set $P(A)$ of unital injective endomorphisms is closed in the Tykhonoff topology. Hence $$\overline{\text{Inn}(A)}\subseteq P(A).$$

Now, let  $\varphi: A \rightarrow A $ be an injective endomorphism and $\varphi(1)=1.$ Let $a_1,\ldots, a_n \in A.$ As above choose a subalgebra $B_1 \subset A $ such that   $1, a_1,\ldots, a_n \in B_1$ and $B_1 \cong M_k(\mathbb{F}).$ Then choose another subalgebra $B_2\subset A$ such that   $$B_1+ \varphi(B_1)\subseteq B_2 \quad \text{and} \quad B_2 \cong M_l(\mathbb{F}).$$  By the Skolem--Noether Theorem (see \cite{Drozd_Kirichenko,SkolemNoether}) there exists an invertible element $$a\in B_2 \quad \text{such that} \quad \varphi(x)=a^{-1}x a \quad \text{for all elements} \quad x\in B_1.$$ This finishes the proof of the Theorem. \end{proof}

\section{Derivations of Tensor Products of Matrix Algebras}

Recall that $I$ is an infinite set and let $\mathbf{A}$ be a tensor product of the kind (\ref{otimes_eq}): $$\mathbf{A}= \otimes_{i\in I} \ A_i,$$ where all algebras $A_i$  are matrix algebras over $\mathbb{F},$ $\dim_{\mathbb{F}}A_i > 1.$ Clearly, $\mathbf{A}$ is a unital locally matrix algebra.

 \begin{lemma} \label{L2}
For any  $i\in I$ the subalgebra $$C=\otimes_{j \neq i} \ A_j$$ is the  centraliser of the subalgebra $A_i$ in $\mathbf{A}.$
\end{lemma}
\begin{proof} We have $$\mathbf{A}= A_i\otimes{}_{\mathbb{F}} \ C.$$ Clearly, $C$ lies in the centraliser  of $A_i$ in $\mathbf{A}.$ Now, suppose that $x\in \mathbf{A}$ and $[A_i,x]=\{0\}.$ Let  $$x= \sum_{k=1}^{n} \ a_k \otimes c_k, \quad  a_1, \ldots, a_n \in A_i , \quad c_1, \ldots , c_n \in C,$$ and the elements $c_1,$ $\ldots,$ $c_n$ are linearly independent. Then for an arbitrary element  $a\in A_i$ we have $$ \sum_{k=1}^{n} \ [ a,a_k] \otimes c_k=0,$$ which implies $$ [a,a_1]= \cdots =[a, a_n]=0.$$ Hence $a_1,$ $\ldots,$ $a_n\in \mathbb{F} \cdot 1_{A_i} $ and $x\in C.$  \end{proof}

 \begin{lemma} \label{L3}
Let $i\in I$ and let $d\in \emph{Der} (\mathbf{A}).$ If $d(A_i)=\{0\}$ then the subalgebra   $$C=\otimes_{j \neq i} \ A_j $$ is $d$--invariant.
 \end{lemma}
\begin{proof} We have $[A_j , C]=\{0 \}.$ Hence $$\big[A_i , d(C)\big]\subseteq d\big(\big[ A_i, C\big]\big)+ \big[d( A_i), C\big]=\{0 \}. $$
By Lemma \ref{L2}, it implies $d(C)\subseteq C.$ \end{proof}

\begin{proof}[Proof of Theorem $\ref{T2_new}$] (1) Let $I=\{i_1,i_2,\ldots \}.$  Let $d\in \text{Der} (\mathbf{A}).$ Since the algebra of inner derivations $\text{Inder}(\mathbf{A}) $  is dense in $\text{Der} (\mathbf{A}),$ by  Theorem \ref{T1_new} (1), it follows that there exists an element $a_1 \in \mathbf{A}$ such that $$\big(d-\text{ad}_{\mathbf{A}}(a_1)\big) (A_{i_1})=\{0\}.$$ There exists  a finite subset $S_1 \subset I$ such that $a_1 \in A_{S_1}.$ By Lemma \ref{L3}, the derivation $d- \text{ad}_{\mathbf{A}}(a_1) $ maps the subalgebra  $$\otimes_{j\neq i_1}  A_j$$ into itself. Arguing as above, we find a finite subset $S_2 \subset I \diagdown \{i_1\}$ and an element $a_2 \in A_{S_2}$ such that $$\big(d-\text{ad}_{\mathbf{A}}(a_1) -\text{ad}_{\mathbf{A}} (a_2) \big) (A_{i_2})=\{0\}$$ and so on. We get a sequence $S_1,$ $S_2,$ $\ldots$ of nonempty finite subsets of $I,$ $$S_n \subset I \ \diagdown \ \{i_1, \ldots, i_{n-1}\},  \quad  n\geq 2,$$ and a sequence of elements $a_n \in A_{S_n},$ $n\geq 1,$ such that $$d=\sum_{n=1}^{\infty} \ \text{ad}_{\mathbf{A}}(a_n).$$ Adding to the subsets $S_1,$ $S_2,$ $\ldots$ all their  nonempty subsets, we get a sparse system $\mathcal{P}$ and $d\in D_{\mathcal{P}}.$ This completes the proof of the part $(1)$ of the Theorem.

\vspace{8pt} (2) Let $I$ be an infinite, not necessarily countable set. Let $\mathcal{P}$ be a sparse system of finite nonempty subsets of $I.$ Choose a subset  $S= \{i_1,\ldots, i_r\}\in \mathcal{P}$ and element  $ a_k \in A_{i_k},$ $ 1\leq k \leq r.$ Let $$a_k =\gamma_k \cdot 1_{A_{i_k}}+ a_k^{0}, \quad \text{where} \quad \gamma_k \in \mathbb{F}, \quad  a_k^{0}\in A_{i_k}^0 .$$ Expanding brackets in the tensor $$  a_1 \otimes \cdots \otimes a_r = \big(\gamma_1 \cdot 1_{A_{i_1}} + a_1^{0}\big) \otimes \cdots \otimes \big(\gamma_r \cdot 1_{A_{i_r}} + a_r^{0}\big)$$ we get  $$ a_1 \otimes \cdots \otimes a_r \in \mathbb{F} \cdot 1 + \sum_{k=1}^r 1 \otimes \cdots \otimes A_{i_k}^0 \otimes \cdots \otimes 1 + \cdots + A_{i_1}^0 \otimes \cdots \otimes A_{i_r}^0.$$ Hence, the space $\text{ad}_{\mathbf{A}} (A_s)$ is spanned by $$\bigcup_{\emptyset \, \neq \, S' \, \subseteq \, S} \ \text{ad}_{\mathbf{A}}(E_{S'}).$$ This implies that an arbitrary element from $D_{\mathcal{P}}$ can be represented as a converging sum $$\sum_k \ \alpha_k \ \text{ad}_{\mathbf{A}} (e_k), \quad \text{where} \quad \alpha_k \in \mathbb{F} \quad \text{and} \quad \{e_k\}_k = \bigcup_{S\in \mathcal{P}} \ E_S.$$ Now we need to show that  $$\sum_k \ \alpha_k \ \text{ad}_{\mathbf{A}} (e_k)=0$$ implies $\alpha_k =0$ for all $k.$

Let $S\in \mathcal{P} ,$ $i\in S,$ $S^0 =S \ \diagdown \ \{i\}.$ An arbitrary element $e$ from $E_S$ can be represented (up to a permutation of tensors) as $$e=e' \otimes e'', \quad \text{where} \quad e' \in E_i, \quad e'' \in E_{S^0}.$$ For an element $a\in A_i$ we have  $[e,a]=[e',a]\otimes e''.$

Fix $i\in I.$ Let $S_1,$ $\ldots,$ $ S_t$ be all subsets from $\mathcal{P}$ that contain $i.$ Let $$S^0_j =S_j \ \diagdown \ \{i\}, \quad 1 \leq j \leq t.$$ If $$e_k = e' \otimes e'', \quad  \text{where} \quad  e' \in E_i \quad \text{and} \quad e'' \in \bigcup_{j=1}^t \ E_{S_j^0}, $$ or $e'' =1$ if all $ S_j=\{i\}, $ then we denote $\alpha_{e',e''}:= \alpha_k.$ For an arbitrary element $a\in A_i$ we have $$ \Big[ \ \sum_k \ \alpha_k e_k, a \ \Big]= \sum \ \alpha_{e',e''} \, [e',a]\otimes e'' =0,  $$ where the summation runs over all $$(e',e'')\in E_i \times \Big( \ \bigcup_{j=1}^t \ E_{S_j^0} \ \bigcup \ \big\{1\big\} \ \Big).$$ Hence $$ \Big[ \ \sum_{e'} \ \alpha_{e',e''}\, e', a \ \Big]=0 \quad  \text{for any} \quad e'' \in  \bigcup_{j=1}^t \ E_{S_j^0} \ \bigcup \ \big\{ 1 \big\}.$$

The element  $\sum_{e'} \alpha_{e',e''} e'$ lies in $A_i^0$ and at the same time it lies in the center of the algebra $A_i.$ Hence $$ \sum_{e'} \alpha_{e',e''} e'=0$$ which implies $\alpha_{e',e''}=0.$ This completes the proof of   Theorem $\ref{T2_new}$.  \end{proof}

In what follows we will need the following Lemma.

\begin{lemma} \label{L6}
Let $A$ be a countable--dimensional  locally matrix algebra. Let $0\neq e\in A$ be an idempotent. Then every derivation of the subalgebra $eAe$ extends to a derivation of $A.$
 \end{lemma}
\begin{proof} Suppose at the first that the algebra $A$ is unital. By the K\"{o}the's Theorem \cite{Koethe},  we can assume that algebra $A$ of the kind (\ref{T_PR0}):  $$A=\otimes_{i=1}^{\infty} A_i,$$ where each factor $A_i$ is a matrix algebra over $\mathbb{F}.$ There exists  $n\geq 1$  such that $e\in A_1 \otimes\cdots \otimes A_n.$ Replacing the first $n$  factors $A_1,$ $\ldots,$ $A_n$ by one  factor $ A_1 \otimes\cdots \otimes A_n$ we can assume that $e\in A_1.$ Then $$ eAe=eA_1 e \otimes \Big(\otimes_{j=2}^{\infty}A_j\Big).$$  Let $d\in \text{Der}(eAe).$ By Theorem \ref{T2_new} (1),  there exists a sparse system $\mathcal{P}$ of nonempty subsets of  the set of positive integers  such that $$ d= \sum_{S\in \mathcal{P}} \ \text{ ad}_{eAe}(a_S), \quad a_S \in (eAe)_S. $$ We have $(eAe)_S\subseteq A_S.$ Since the system $\mathcal{P}$ is sparse it follows that the infinite sum  $$ \sum_{S\in \mathcal{P}} \ \text{ ad}_{A}(a_S)$$ converges in the Tykhonoff topology on $\text{Map}(A,A)$ to a derivation that extends $d.$

Suppose now that  $A$ is a countable--dimensional non unital locally matrix algebra. There exists a sequence of idempotents $e_1=e,$  $e_2,$ $\ldots$ such that
\begin{equation}\label{EQ_3stars}
 e_i A e_i \subset e_{i+1} A e_{i+1}, \quad i\geq 1, \quad \text{and} \quad \bigcup_{i\geq 1} \ e_i A e_i=A.
\end{equation}
Let $d\in \text{Der} (eAe).$ Remark, that for any idempotent $e'\in A$ the subalgebra $e'Ae'$ is a unital locally matrix subalgebra. So, by the unital case of this Lemma (see above), there exist derivations $d_i\in \text{Der} (e_iAe_i)$ such that $d_{i+1}$ extends $d_i,$ $i\geq 1,$ and $d_1 =d.$ The derivation $$\bigcup_{i=1}^{\infty} \ d_i  \in \text{Der}(A)$$ extends the derivation $d.$  \end{proof}

\section{The Lie Algebra  of  Outer Derivations is Not Locally Finite--Dimensional}

Let $\mathbb{N}$ be the set of positive integers and let $M_{\infty}(\mathbb{F})$ be the algebra of all infinite $\mathbb{N} \times \mathbb{N}$ finitary matrices over $\mathbb{F},$ that is matrices that contain finitely many nonzero entries.

\begin{lemma} \label{L4_new!}
Let $A$ be a countable--dimensional non unital locally matrix algebra such that for every idempotent $e\in A$ we have $\dim_{\mathbb{F}} eAe < \infty.$ Then $A\cong M_{\infty}(\mathbb{F}).$
 \end{lemma}
\begin{proof} Since the algebra $A$ is countable--dimensional and non unital there exists a sequence of idempotents $e_1,$ $e_2,$ $\ldots$ such that  (\ref{EQ_3stars}) holds.

We claim that each subalgebra $e_i A e_i$ is isomorphic to a matrix algebra over $\mathbb{F}.$ Indeed, since $\dim_{\mathbb{F}} e_i A e_i < \infty$ there exists a subalgebra $A' \subset A$ such that $e_i A e_i \subseteq A'$ and $A'$ is isomorphic to a matrix algebra. Let $$\psi:A'\rightarrow M_t(\mathbb{F})$$ be an isomorphism. Let $n_i$ be the range of the matrix $\psi(e_i)$ in $M_t(\mathbb{F}).$ Then $$e_i A e_i \ \cong \ \psi (e_i)\, M_t(\mathbb{F}) \, \psi(e_i) \ \cong \ M_{n_i}(\mathbb{F}).$$ Let $$\text{id}_{i,i+1}: e_i A e_i \rightarrow e_{i+1} A e_{i+1}$$ be the embedding homomorphism, $i\geq 1.$ It is easy to see that there exists a sequence of isomorphisms $$\varphi_i :M_{n_i}(\mathbb{F}) \rightarrow e_i A e_i, \quad n_i \geq 1, \quad i \geq 1,$$ such that the embeddings $$ \varphi_{i+1}^{-1} \circ \text{id}_{i,i+1}\circ \varphi_i \quad \text{of} \quad M_{n_i}(\mathbb{F}) \quad \text{into} \quad M_{n_{i+1}}(\mathbb{F}) $$ is diagonal, that is
$$a \rightarrow \left(
\begin{array}{cc}
a & 0 \\
0 & 0 \\
\end{array}
\right), \quad a \in M_{n_i}(\mathbb{F}). $$

The algebra $A$ is isomorphic to the direct limit of matrix algebras $M_{n_i}(\mathbb{F})$ with diagonal embeddings, that is, to $M_{\infty}(\mathbb{F}).$  \end{proof}

\begin{proof}[Proof of Theorem  $\ref{T3}$] Suppose at the first that the algebra $A$ is unital. Then by the K\"{o}the's Theorem~\cite{Koethe}, the algebra $A$ of the kind (\ref{T_PR0}). We will assume that $$A=\otimes_{i=1}^{\infty} A_i, \quad A_i \cong M_{n_i} (\mathbb{F}), \quad i\geq 1, \quad n_i\geq 2.$$ The algebras $A_i$ are embedded in $A$  via $$u_i : a  \mapsto \underbrace{1 \otimes \cdots  \otimes 1  \otimes a}_i \otimes 1 \otimes  \cdots, \ \ \ \ a \in A_i.$$ Let $$e_{pq}(i):=u_i (e_{pq}), \quad 1 \leq p,q \leq n_i,$$ denote the image of the matrix unit $e_{pq}\in A_i \cong M_{n_i} (\mathbb{F}).$  Since the  images of $A_i$ and $A_j,$ $i\neq j,$ commute in $A,$ we have $$[\,e_{pq}(i),e_{rs}(j)\,]=0 \quad \text{ for } \quad i\neq j.$$

Consider  the following derivations of the algebra $A:$ $$ z=\sum_{i=1}^{\infty}\ \text{ad}\big(e_{12}(i) e_{11}(i+1)\big)\in \text{ad}(A_{[1,2]})+\text{ad}(A_{[2,3]})+\cdots$$ and $$ y_k =\sum_{j=1}^{\infty}\ \text{ad}\big(e_{12}(j)  \cdots e_{12}(j+k-1)\big)\in   \text{ad}(A_{[1,k]})+\text{ad}(A_{[2,k+1]})+\cdots ,$$ where  $k\geq 1$ and $[t,l]$ is the integer segment, $[t,l]=\{t, t+1,\ldots, l\},$ $1\leq t\leq l.$

We claim that $[z,y_k]=y_{k+1}$ for any $k\geq 1.$ Indeed, $$[z,y_k]=\sum_{i,j} \ \big[\,\text{ad}\big(e_{12}(i) e_{11}(i+1)\big),\text{ad}\big(e_{12}(j)\cdots e_{12}(j+k-1)\big)\, \big]= $$ $$ \sum_{i,j} \ \text{ad}\big(\,[\, e_{12}(i) e_{11}(i+1),e_{12}(j)\cdots e_{12}(j+k-1)\, ]\,\big).$$

If $\{i, i+1\} \cap \{j,  \ldots, j+k-1\}= \emptyset$ then each factor of $e_{12}(i) e_{11}(i+1)$  commutes with each factor of $e_{12}(j)\cdots e_{12}(j+k-1).$

If $i \in \{j, \ldots, j+k-1\}$ then $$e_{12}(i) e_{11}(i+1)\cdot e_{12}(j)\cdots e_{12}(j+k-1)=$$ $$e_{12}(j)\cdots e_{12}(j+k-1)\cdot e_{12}(i) e_{11}(i+1) =0 $$  since $ e_{12}(i)^2 =0.$

It remains to consider only one case: $j=i+1.$ We have  $$e_{12}(i) e_{11}(i+1)\cdot e_{12}(i+1)\cdots e_{12}(i+k)=e_{12}(i) e_{12}(i+1)\cdots e_{12}(i+k).$$ Multiplying these elements in the other order we get $$e_{12}(i+1) \cdots e_{12}(i+k)\cdot e_{12}(i) e_{11}(i+1)=0$$ since $e_{12}(i+1)  e_{11}(i+1)=0.$ Finally,
$$[z,y_k]=\sum_{i=1}^{\infty} \ \text{ad}\big(\,[\,e_{12}(i) e_{11}(i+1),e_{12}(i+1)\cdots e_{12}(i+k)\,]\,\big)= $$ $$\sum_{i=1}^{\infty} \ \text{ad}\big(e_{12}(i) e_{12}(i+1)\cdots e_{12}(i+k)\big)= y_{k+1}.$$

The Lie subalgebra of $\text{Der}(A)$ generated by  elements $z$ and $y_1$ contains all elements  $y_k,$ $k\geq 1.$

Let us show that  derivations $y_k,$ $k\geq 1,$ are linearly independent modulo $\text{Inder} (A).$ Recall that in each algebra $A_i$ we choose a subspace $A_i^0$ so that (\ref{eq5}) holds. Choose a subspace $A_i^{0}$ containing $e_{12}(i) $  and a basis $E_i$ in $A_i^{0}$ such that $e_{12}(i) \in E_i.$  Then $$e_{12}(i) \cdots e_{12}(i+k-1)\in E_{[i,\,i+k-1]}.$$ Suppose that $\alpha_1 y_1 + \cdots +\alpha_k y_k \in  \text{Inder} (A)$ and $\alpha_1, \ldots,\alpha_k \in \mathbb{F}.$  Then there exists $p\geq 1$ such that $$\alpha_1 y_1 + \cdots +\alpha_k y_k \in \text{ad}_A(A_{[1,\,p]}).$$ Without loss of generality, we will assume that $k\leq p.$

Consider a sparse system $\mathcal{P}$ that consists of intervals $[i,i+p-1],$ $i\geq 1,$ and all their nonempty subsets. Let $E$ denote the topological basis of the vector space $D_{\mathcal{P}}$ that corresponds to bases $E_i$ of  subspaces $A_i^{0};$ see Theorem~\ref{T2_new}~(2). We have
\begin{equation}\label{ad1}
 \alpha_1 y_1 + \cdots +\alpha_k y_k =\sum_{1\leq j \leq k, \ 1\leq i <\infty}   \alpha_j \text{ad}_A\big(e_{12}(i) \cdots  e_{12}(i+j-1)\big).
\end{equation}
The operators $\text{ad}_A(e_{12}(i) \cdots e_{12}(i+j-1))$  are distinct elements of the basis $E.$ If at least one coefficient $\alpha_j,$ $1\leq j \leq k,$ is not equal to $0,$ then the sum (\ref{ad1}) contains infinitely many basis elements from $E$ with nonzero coefficients. Hence, it can not be equal to a finite linear combination of basic elements from $E.$ Every element from $\text{ad}_A(A_{[1,p]})$ is a finite linear combination of basis elements. Therefore $\alpha_1=0,$ $\ldots,$ $\alpha_k=0.$ This proves the claim.

We showed that the Lie subalgebra of $\text{Der}(A)$ generated by  derivations $z$ and $y_1$  is infinite--dimensional module $\text{Inder}(A).$ This completes the proof of the Theorem in the case when the algebra $A$ is unital.

Now, let $A$ be a countable--dimensional  non unital locally matrix algebra. Suppose that there exists an idempotent $e\in A$ such that the unital  algebra $eAe$ is infinite--dimensional. We have shown that there exist derivations  $z$ and $y_1$ of the algebra $eAe$  such that the derivations $$y_k=\big[\underbrace{z,[z,\ldots,[z}_{k-1},y_1]\ldots]\,\big], \quad k\geq 1,$$ are linearly independent  module $\text{Inder}(eAe).$

By Lemma \ref{L6}, there exist derivations $\tilde{z}, \tilde{y}_1\in \text{Der}(A)$ that extend $z$ and $y_1$ respectively. Let us show that the derivations $$\tilde{y}_k=\big[\underbrace{\tilde{z},[\tilde{z},\ldots,[\tilde{z}}_{k-1},\tilde{y}_1]\ldots]\,\big], \quad  k\geq 1,$$ are linearly independent module $\text{Inder}(A).$

Suppose that $$d=\alpha_1 \tilde{y}_1 + \cdots +\alpha_n \tilde{y}_n \in  \text{Inder} (A), \quad \alpha_1,\ldots,\alpha_n\in \mathbb{F}.$$ We will show that in this case $\alpha_1 y_1 + \cdots +\alpha_n y_n \in  \text{Inder} (eAe).$ The derivation  $\tilde{y}_i$ extends the derivation $y_i.$ Hence, the subalgebra $eAe$ is invariant with respect to $d.$ Since $d\in \text{Inder} (A) $ then  there exists an element $u\in A$ such that $d(x)=[u,x] $ for an arbitrary element $x\in A.$

Consider the Peirce decomposition $$u= eue+(1-e)ue + eu(1-e)+(1-e)u(1-e),$$ where $1$ is a formal unit. For an arbitrary element $x\in eAe$ we have $$[u,x]= [eue,x]+(1-e)uex-xeu(1-e).$$ The inclusion $[u,x]\in eAe$ implies $[u,x]=[eue,x].$

We showed that the restriction of the derivation $d$ to $eAe$ is an inner derivation. Hence $$\alpha_1 y_1 + \cdots +\alpha_n y_n \in  \text{Inder} (eAe),$$ which implies $\alpha_1=\cdots=\alpha_n =0.$

By Lemma \ref{L4_new!}, if for an arbitrary idempotent $e\in A$ the subalgebra $eAe$ is finite--dimensional, then $A\cong M_{\infty}(\mathbb{F}).$ Thus, it remains to verify that the Lie algebra of outer derivations $\text{Outder}(M_{\infty}(\mathbb{F}))$ is not locally finite--dimensional.

Infinite matrices $$ z= \sum_{i=1}^{\infty} \ e_{2i, \, 2i+2} \quad \text{and} \quad y_k =\sum_{i=1}^{\infty} \ e_{2i, \, 2i+2k-1}, \quad k\geq 1, $$ are not finitary, but $$ [\,z,M_{\infty}(\mathbb{F})\,]\subseteq M_{\infty}(\mathbb{F}) \quad \text{and} \quad [\,y_k ,M_{\infty}(\mathbb{F})\,]\subseteq M_{\infty}(\mathbb{F}), \quad k\geq 1.$$ We have $[z, y_k]=y_{k+1},$ $k\geq 1.$ The subalgebra generated by derivations $\text{ad}(z), \text{ad}(y_1) \in \text{Der}(M_{\infty}(\mathbb{F}))$ contains all derivations $\text{ad}(y_k),$ $k\geq 1.$ It is easy to see that the derivations $\text{ad}(y_k),$ $k\geq 1,$ are linearly independent modulo $\text{Inder}(M_{\infty}(\mathbb{F})).$ It completes the proof of  Theorem $\ref{T3}$. \end{proof}

\section{Automorphisms and Unital Injective Endomorphisms}

\begin{proof}[Proof of Theorem $\ref{T4_new}$] Let $\varphi: A \rightarrow A$ be an injective endomorphism of the countable--dimensional unital locally matrix algebra algebra  (\ref{T_PR0}), $\varphi(1)=1.$ There exists a finite subset $S_1 \subset \mathbb{N}$ such that $\varphi  (A_1)  \subseteq  A_{S_1}.$ Applying the Skolem--Noether Theorem (see \cite{Drozd_Kirichenko,SkolemNoether}), as we did in the proof of Theorem~\ref{T1_new} we find an invertible  element $a_1 \in  A_{S_1}$ such that  $$ \varphi(x)  =  a^{-1}_1  x  a_1 \quad \text{for all elements} \quad  x  \in   A_1.$$ Let $\hat{a_1}$ be the automorphism of conjugation by the element $a_1.$ So, $\hat{a_1}  \in   \text{H}_1. $ Let $\varphi_1\in \mathcal{X}_1$ be a representative of the coset $\hat{a_1}\text{H}_2.$  The embedding $\psi_1   =  \varphi_1^{-1}  \varphi $  fixes all elements in the subalgebra $A_1.$

For an arbitrary element  $ a  \in  \otimes_{j\geq 2}  A_j $ we have $$\{0\}  =  \psi_1  (  [  A_1 ,  a   ]  )  =  [  \psi_1(A_1)  ,  \psi_1(a)  ]  =  [   A_1  ,  \psi_1(a)  ].$$ Hence, the element  $\psi_1(a)$ lies in the centralizer of  $A_1.$ By Lemma~\ref{L2}, $$ \psi_1(a)  \in  \otimes_{j \geq 2} \ A_j. $$ We showed that $\psi_1$ is an embedding of the algebra $\otimes_{j\geq 2}  \ A_j $ into itself.

Arguing as above, we find an automorphism   $ \varphi_2  \in  \mathcal{X}_2 $ such that $ \varphi_2^{-1} \ \psi_1$ fixes all elements in the subalgebra $A_2,$ and so on. As a result, we represent $\varphi$ as an infinite product  \begin{equation*}\label{bigotimes_2}\varphi  =  \varphi_1 \, \varphi_2 \, \cdots,  \quad  \varphi_i  \in  \mathcal{X}_i, \quad i  \geq  1. \end{equation*}

Now suppose that $$ \varphi_1\, \varphi_2 \, \cdots =  \varphi\,'_1 \, \varphi\,'_2 \, \cdots  , \quad \text{where} \quad \varphi\,'_i \in \mathcal{X}_i, \quad i \geq 1.$$ Applying both sides to elements from $A_1,$ we see that  $$\varphi_1  \big| \, \raisebox{-5pt}{$A_1$} = \varphi\,'_1  \big| \, \raisebox{-5pt}{$A_1$} .$$ Let $\varphi_1,$ $\varphi\,'_1$ be conjugations by invertible elements $a,$ $b$ respectively. Then the element $a^{-1}b$ lies in the centralizer of  $A_1,$ hence in $\otimes_{j> 1} \ A_j.$ So,  $\varphi_1^{-1}  \varphi\,'_1  \in  \text{H}_2$ and $\varphi_1  =  \varphi\,'_1.$ This implies  $$ \varphi_2 \, \varphi_3 \, \cdots  =  \varphi\,'_2 \, \varphi\,'_3 \, \cdots  .$$ Arguing as above, we see that $ \varphi_2 =  \varphi\,'_2,$ $ \varphi_3 =  \varphi\,'_3$ and so on. \end{proof}

\begin{proof}[Proof of Theorem $\ref{T5_new}$] Suppose that the sequence of automorphisms (\ref{EQ_integrable}) is integrable. Then for an arbitrary positive integer $p\geq 1$ the subspace spanned by $$\varphi_i^{-1} \cdots \varphi_1^{-1}(A_p),  \quad i \geq \ 1,$$ is finite--dimensional. Hence, there exists positive integer $q \geq 1$ such that $$ \varphi_i^{-1} \cdots \varphi_1^{-1}(  A_p  ) \ \subseteq \ A_{[1,\, q]} \quad \text{for any} \quad i  \geq  1.$$ This inclusion is equivalent to  $$ A_p \ \subseteq \ \varphi_1 \cdots \varphi_i \big(  A_{[1,\,q]}  \big), \quad i  \geq  1.$$ For $i=q$ we have $$ \varphi_1 \cdots \varphi_q (  A_{[1,\,q]} )  =  \varphi (  A_{[1,\,q]} ),$$ and therefore $$A_p \ \subseteq \ \varphi (  A_{[1,\,q]} ).$$ We showed that the injective endomorphism  $\varphi$ is surjective, hence an automorphism.

Now suppose that the injective endomorphism $\varphi$ is surjective. Then for an arbitrary $p\geq 1$ there exists $q \geq 1$ such that $$A_{[1,\, p]} \ \subseteq \ \varphi (  A_{[1,\,q]} )  =  \varphi_1 \cdots \varphi_i (  A_{[1,\,q]} ) \quad \text{for} \quad i  \geq  q.$$ Hence $$ \varphi_i^{-1} \cdots \varphi_1^{-1}(  A_{[1,\, p]} ) \  \subseteq \ A_{[1,\, q]} \quad \text{for} \quad i  \geq  q.$$ It implies that the subspace spanned by $$ \varphi_i^{-1} \cdots \varphi_1^{-1}(  A_{[1,\, p]}  ), \quad i  \geq  1, $$ is finite--dimensional, hence the sequence (\ref{EQ_integrable}) is integrable.  \end{proof}

\begin{proof}[Proof of Example $\ref{example1}$] For an  arbitrary subalgebra $ A_{i_1} \otimes  \cdots  \otimes A_{i_r}$ of the algebra $A$  and an arbitrary positive integer $j \geq 1$ we have $$a_j \cdots a_1(  A_{i_1}  \otimes  \cdots  \otimes  A_{i_r} )  a_1^{-1} \cdots a_j^{-1} = A_{i_1}  \otimes  \cdots   \otimes  A_{i_r}. $$ In particular, the subspace spanned by   $$\hat{a_j}^{-1} \cdots\hat{a_1}^{-1}(  A_{i_1}  \otimes  \cdots   \otimes  A_{i_r} )  , \quad  j  \geq  1, $$ is finite--dimensional and the sequence  $\hat{a_i}^{-1},$ $i \geq 1,$ is integrable. \end{proof}

\begin{proof}[Proof of Example $\ref{example2}$]  Recall that   $ a_i   =  e_{11}(i)  e_{12}(i+1)$ and the automorphism $\phi_i$ is a conjugation by  $(1+a_i)^{-1},$ $i\geq 1.$ Let $a_0 =e_{12}(1).$ Obviously, $(1+a_i)^{-1}=1-a_i$ for $i\geq 0.$ We claim that the sequence (\ref{EQ_4-0}) is not integrable. We will use induction on $i$ to prove that
\begin{equation}\label{stres}
(1  +  a_i)  \cdots  (1  +  a_1)  e_{12}(1)  (1  +  a_1)^{-1}  \cdots  (1  + a_i)^{-1}  =
\end{equation}
$$ e_{12}(1)  + e_{12}(1)  e_{12}(2)  +  \cdots  +  e_{12}(1)  e_{12}(2) \cdots  e_{12}(i+1).$$
For $i=0$ the assertion is obvious. Consider the element $$ \big(1  +  a_{i+1}\big) \ \Big( \ \sum_{k=1}^{i+1} \ e_{12}(1) \cdots  e_{12}(k) \ \Big) \ \big(1  -  a_{i+1}\big).$$ For an arbitrary  $k,$ $1\leq k \leq i+1,$ we have $$a_{i+1}  e_{12}(1)  \cdots  e_{12}(k)  a_{i+1}  = $$ $$ e_{11}(i+1)  e_{12}(i+2)  e_{12}(1)  \cdots  e_{12}(k)  e_{11}(i+1)  e_{12}(i+2)  =  0,$$ since $$ e_{12}(i+2)^2  =  0.$$ Hence $$ (1  +  a_{i+1}) \ \Big(  \sum_{k=1}^{i+1} \ e_{12}(1)  \cdots  e_{12}(k) \Big) \ (1  -  a_{i+1})  = $$ $$ \sum_{k=1}^{i+1} \ e_{12}(1)  \cdots  e_{12}(k)  +  \Big[ \ a_{i+1}  , \, \sum_{k=1}^{i+1} \ e_{12}(1)  \cdots  e_{12}(k)  \ \Big] .$$
Since elements from different tensor factors commute, we get for  $1\leq k \leq i$ $$ \big[ \ e_{11}(i+1)   e_{12}(i+2)  ,    e_{12}(1)  \cdots  e_{12}(k) \ \big]  =   0. $$ For $k=i+1$ $$ \big[ \ e_{11}(i+1)   e_{12}(i+2)  ,    e_{12}(1)  \cdots  e_{12}(i)  e_{12}(i+1) \ \big]  = $$ $$  e_{12}(1)  \cdots  e_{12}(i) \ \big[ \ e_{11}(i+1)  ,   e_{12}(i+1) \ \big]  e_{12}(i+2)  =  $$ $$  e_{12}(1)  \cdots  e_{12}(i+2).$$ So, (\ref{stres}) holds. Since the elements $ e_{12}(1)  \cdots  e_{12}(i),$ $i  \geq  1, $ are linearly independent in the algebra $A,$ we conclude that the subspace spanned by the elements  $$\phi_i \cdots\phi_1( e_{12}(1) )  = $$ $$ (1   +  a_i)  \cdots  (1  +  a_1)  e_{12}(1)  (1  +  a_1)^{-1}  \cdots  (1  +  a_i)^{-1}, \quad i  \geq  1,$$  is infinite--dimensional. Hence, the sequence (\ref{EQ_4-0}) is not integrable.

By Theorem \ref{T5_new}, the injective endomorphism  $\phi = \phi_1 \phi_2 \cdots $ is not surjective. Hence, the subalgebra $B=\phi(A)$ is isomorphic to  $A$ and $B \subsetneqq A.$ This is another proof of Theorem 10 from \cite{Kurosh}. \end{proof}

In the next chapter we will use the following Lemma.

\begin{lemma} \label{L5_new!} Let $A$ be a countable--dimensional locally matrix algebra. Let $e\in A$ be an idempotent. Then every automorphism of $eAe$ extends to an automorphism of $A.$
 \end{lemma}

\begin{proof} At first, let us assume that the algebra $A $ is unital. Let $\varphi$ be an automorphism of the subalgebra $eAe.$ If the automorphism $\varphi$ is inner then there exists an invertible element $x_e$ in the  subalgebra $eAe$ such that $$\varphi(a) = x_e^{-1}a x_e \quad \text{for all elements} \quad a\in eAe.$$  In this case, the element $x=x_e+ (1-e)$ is invertible in $A.$ The automorphism of conjugation $a\mapsto x^{-1}a x, $ $a\in A,$ extends the automorphism~$\varphi.$

Let $\varphi$ not be an  inner  automorphism. Then choose  subalgebras $A_1 \subseteq A_2 \subset A$ such that $1,$ $e\in A_1$ and $A_1\cong M_m(\mathbb{F})$ for some $m\geq 1,$ and $$\varphi(eA_1 e)\subseteq e A_2 e \quad \text{and} \quad  A_2 \cong M_n(\mathbb{F}) \quad \text{for some} \quad n\geq 1.$$ Let $\varphi':=\varphi|_{e A_1 e}$ be the restriction of $\varphi$ to the subalgebra $e A_1 e,$ so that $$ \varphi': eA_1 e \rightarrow \varphi(eA_1 e) , \quad \text{and} \quad \varphi' : e\mapsto e.$$ By the Skolem--Noether Theorem (see \cite{Drozd_Kirichenko,SkolemNoether}) there exists an invertible element $x_e \in eA_2 e$ such that $$\varphi'(a)=x_e^{-1}a x_e \quad \text{for all elements} \quad a\in eA_1 e.$$
Now, let us consider the  automorphism $$\psi': eAe \rightarrow eAe, \quad a\mapsto x_e^{-1}a x_e.$$ As we have shown above, the inner automorphism $\psi'$ of the subalgebra $eAe$ extends to some automorphism $\psi$ of the algebra $A.$ So, it is sufficient to show that the automorphism $\psi'^{-1}\circ \varphi\in \text{Aut}(eAe)$ extends to some automorphism $\tilde{\varphi}$ of $A.$ Then the   automorphism $\varphi$ extends to the automorphism $\psi \circ \tilde{\varphi}$ of $A.$

Let $\varphi_1:=\psi'^{-1}\circ \varphi.$ The composition $\varphi_1$ leaves every element from $eA_1 e$ fixed. Let $C$ be the centralizer of the subalgebra $A_1$ in $A.$ Then $$A=A_1\otimes_{\mathbb{F}} C \quad \text{and} \quad eAe = eA_1 e \otimes_{\mathbb{F}} C.$$ Since the subalgebra $e\otimes_{\mathbb{F}} C$ is the centralizer of $eA_1e$ in $eAe$ it follows that $e\otimes_{\mathbb{F}} C$ is invariant with respect to $\varphi_1.$ Hence, there exists an automorphism $\theta\in \text{Aut}(C)$ such that $$\varphi_1(a\otimes c)=a \otimes \theta(c) \quad \text{for arbitrary elements} \quad a\in eA e, \quad c \in C.$$ Now, the automorphism $$\tilde{\varphi}: A\rightarrow A,  \quad \tilde{\varphi}(a\otimes c) = a\otimes \theta(c), \quad a\in A_1, \quad c\in C,$$ extends $\varphi_1.$

We have proved the Lemma in the case when the algebra $A$ is unital.

Now suppose that the algebra $A$ is not unital. Then there exists a sequence of idempotents  $e_i \in A,$ $i\geq 1,$ such that $$e_1=e, \quad e_1 A e_1 \subset e_2 A e_2 \subset \cdots \quad \text{and} \quad \bigcup_{i\geq 1} \ e_i A e_i = A. $$ By what we proved above, there exists a sequence of automorphisms $$\varphi_i \in \text{Aut}(e_i A e_i), \quad \varphi_1 = \varphi \quad \text{and} \quad \varphi_{i+1} \big| \, \raisebox{-5pt}{$e_i A e_i$} = \varphi_i. $$ The union $$\tilde{\varphi} = \bigcup_{i\geq 1} \ \varphi_i$$ is an automorphism of $A$ that extends $\varphi.$  \end{proof}

\section{Dimensions of Lie Algebras of Derivations and Orders of Groups of Automorphisms}

In the proofs of Theorems \ref{T6_new_P1}, \ref{T7_new_T2} we will use the following nontrivial Theorem from Linear Algebra, that is due to P.~Erd\"{o}s and I.~Kaplanskiy; see \cite{Jacobson}\footnote{The author is grateful to V.~V.~Sergeichuk for this reference}.

Let $V$ be a vector space over a field $\mathbb{F}$ of infinite dimension $d.$ Let $V^{*}$ be the dual space, that is  the space of all functionals $V\rightarrow \mathbb{F}.$

\begin{theorem}[P.~Erd\"{o}s, I.~Kaplanskiy] \label{Erdos_Kaplanskiy} $$\dim_{\mathbb{F}}  V^{*} = |\,\mathbb{F}\,|^{d}.$$
\end{theorem}

\begin{proof}[Proof of Theorem $\ref{T6_new_P1}$] Consider the vector space $\text{Lin} (\mathbf{A})$ of all linear transformations $\mathbf{A}\rightarrow \mathbf{A}.$ Obviously, $$ \dim_{\mathbb{F}} \text{Der}(\mathbf{A}) \leq \dim_{\mathbb{F}} \text{Lin}(\mathbf{A}) \leq | \text{Lin}(\mathbf{A})|. $$ The dimension of the algebra $\mathbf{A}$ is equal to $|I|.$ The cardinality of the set $\text{Lin} (\mathbf{A})$ does not exceed the  cardinality of all $(I\times I)$-matrices over the field $\mathbb{F},$ the latter being  equal to  $$ |\text{Map}(I\times I,\mathbb{F})| = |\mathbb{F}|^{|I\times I|}=|\mathbb{F}|^{|I|},$$ since $|I|^2 = |I|;$ see \cite{Rubin}. We proved that  $$ \dim_{\mathbb{F}} \text{Der}(\mathbf{A})\leq  |\mathbb{F}|^{|I|}.$$ For an arbitrary index  $i\in I$ choose an element $0 \neq a_i \in A_i^{0}.$ Let $\mathcal{P}$ be the system of all one--element subsets of $I.$ Clearly, the system $\mathcal{P}$ is sparse.

For an arbitrary mapping $f:I\rightarrow \mathbb{F}$ consider the derivation $$ d_f=\sum_{i\in I} \ f(i) \text{ ad}_{\mathbf{A}}(a_i) \in D_{\mathcal{P}}. $$ By Theorem \ref{T2_new} (2), the mapping $f\rightarrow d_f$ is an embedding of the vector space  $\text{Map} (I,\mathbb{F})$ into the vector space $\text{Der}(\mathbf{A}).$ By the Erd\"{o}s--Kaplanskiy Theorem (see Theorem \ref{Erdos_Kaplanskiy}) we have  $$ \dim_\mathbb{F} \text{Map} (I,\mathbb{F}) =|\mathbb{F}|^{|I|}.$$ Hence  $$|\mathbb{F}|^{|I|}\leq \dim_{\mathbb{F}} \text{Der} (\mathbf{A}),  \quad \text{and finally} \quad  \dim_{\mathbb{F}} \text{Der} (\mathbf{A})=|\mathbb{F}|^{|I|} .$$ The dimension of the Lie algebra  $\text{Inder}(\mathbf{A})$ is equal to  $|I|,$  $|I|<|\mathbb{F}|^{|I|}.$ This implies that the equality (\ref{eq1}) holds. \end{proof}

\begin{proof}[Proof of Theorem  $\ref{T7_new_T2}$] Let $A$ be a countable--dimensional locally matrix algebra over a field $\mathbb{F}.$ Assume at first, that the algebra $A$ is unital.  Then by the K\"{o}the's Theorem \cite{Koethe}, the algebra $A$ is isomorphic to a countable tensor product  of finite--dimensional  matrix algebras. Now, the Theorem immediately follows from  Theorem \ref{T6_new_P1}.

Suppose now that the algebra $A$ is not  unital. As above,  $$\dim_\mathbb{F}\text{Der}(A)\leq \dim_\mathbb{F}\text{Lin}(A)\leq |\text{Lin}(A)|= |\mathbb{F}|^{\aleph_0}.$$

Let $e$ be an idempotent of the algebra $A$ such that the subalgebra $eAe$ is infinite--dimensional. By Lemma \ref{L6}, every derivation of the subalgebra $eAe$ extends to a derivation of the algebra  $A.$  The algebra $eAe$ is countable--dimensional and unital. From what we proved above, it follows that  $$|\mathbb{F}|^{\aleph_0}=\dim_\mathbb{F}\text{Der}(eAe)\leq \dim_\mathbb{F}\text{Der}(A).$$ We proved that in this case  $$\dim_\mathbb{F}\text{Der}(A)= |\mathbb{F}|^{\aleph_0}.$$

Now, it remains only to consider the case when $\dim_\mathbb{F} eAe < \infty$ for all idempotents $e\in A.$ By Lemma \ref{L4_new!}, in this case $A\cong M_{\infty}(\mathbb{F}).$ For an arbitrary mapping $f:\mathbb{N}\rightarrow \mathbb{F}$ consider the infinite diagonal matrix $$d_f = \text{diag}(0, f(1), f(2),\ldots).$$ The matrix $d_f$ is not necessarily finitary, but $$[d_f, M_{\infty}(\mathbb{F})]\subseteq M_{\infty}(\mathbb{F}).$$ Hence, $$\text{ad}_{M_{\infty}(\mathbb{F})}(d_f):x \mapsto [d_f,x],$$ is a derivation of the algebra $M_{\infty}(\mathbb{F}).$ The mapping $$f\rightarrow \text{ad}_{M_{\infty}(\mathbb{F})}(d_f)$$ is an embedding of vector spaces $$\text{Map}(\mathbb{N},\mathbb{F}) \rightarrow \text{Der}(M_{\infty}(\mathbb{F})).$$  By the Erd\"{o}s--Kaplanskiy Theorem (see Theorem \ref{Erdos_Kaplanskiy}), $$\dim_{\mathbb{F}}\text{Map}(\mathbb{N},\mathbb{F})= |\mathbb{F}|^{\aleph_0}.$$ Hence $$|\mathbb{F}|^{\aleph_0} \leq \dim_{\mathbb{F}} \text{Der}(M_{\infty}(\mathbb{F})). $$ We proved that $$\dim_{\mathbb{F}} \text{Der}(M_{\infty}(\mathbb{F}))= |\mathbb{F}|^{\aleph_0}. $$ Since the Lie algebra $\text{Inder} (A)$ is countable--dimensional and $\aleph_0 < |\mathbb{F}|^{\aleph_0}, $ it follows that $$\dim_{\mathbb{F}} \text{Outder}(M_{\infty}(\mathbb{F}))= |\mathbb{F}|^{\aleph_0}. $$  \end{proof}

\begin{proof}[Proof of Theorem  $\ref{T8_new_P1}$] As above, we start with the case when the algebra $A$ is unital. So,  $$A\cong\otimes_{i=1}^{\infty} \ A_i, \quad A_i \cong M_{n_i}(\mathbb{F}), \quad n_i \geq 2, \quad i\geq 1. $$ Let $$PGL(n_i,\mathbb{F})=GL(n_i,\mathbb{F}) \, \diagup \, \raisebox{-2pt}{$\mathbb{F}^{*}$} $$ denote the projective linear group. Consider the set $\mathbf{F}$ of mappings $$f: \mathbb{N} \rightarrow \bigcup_{i=1}^{\infty} PGL(n_i,\mathbb{F}) $$ such that $f(i)\in PGL(n_i,\mathbb{F})$ for all $i \in \mathbb{N}.$ It is easy to see that  $ | \mathbf{F}  |  =  |  \mathbb{F}  |\,^{\aleph_0}. $ For an invertible element $a\in A$ let $\hat{a}$ denote the automorphism of conjugation by $a.$ In  Example \ref{example1}, we showed that the sequence of inner automorphism $ \hat{f(i)}^{-1}, i  \geq  1,$  is integrable. Hence by Theorem \ref{T5_new}, the infinite product $$ \varphi_f  =  \hat{f(1)}  \hat{f(2)}  \cdots $$ is an automorphism of the algebra $A.$

Let us show that the mapping   $f\rightarrow \varphi_f$ is injective. Let $f, \, g \in \mathbf{F}$ and $\varphi_f =\varphi_g .$ Applying  automorphisms $\varphi_f,$ $\varphi_g $ to $A_1,$ we see that $$\hat{f(1)}\big|_{A_1} = \hat{g(1)}\big|_{A_1}. $$ Hence $f(1)=g(1).$ Therefore $\hat{f(2)} \hat{f(3)} \cdots = \hat{g(2)} \hat{g(3)} \cdots  .$ Applying both sides to $A_2,$ we get $f(2)=g(2)$ and so on. So, $| \text{Aut} (A) |  \geq  |  \mathbb{F}  |^{\aleph_0}. $

On the other hand, $$| \text{Aut} (A)  |  \leq  | \text{Lin}(A) |  =   |  \mathbb{F}  |^{\aleph_0}. $$ We proved that for a unital algebra $A$ \ $| \text{Aut} (A)  |   =   |  \mathbb{F}  |^{\aleph_0}.$

Now, let the algebra $A$ be not unital. Suppose that $A$ contains an idempotent $e$ such that \ $\dim_{\mathbb{F}} eAe = \aleph_0.$ The algebra $eAe$ is unital. Hence by what we proved above and by Lemma \ref{L5_new!}, $$ |  \mathbb{F}  |^{\aleph_0}=|\text{Aut}(eAe)| \leq |\text{Aut}(A)| \leq |\text{Lin}(A)| = |  \mathbb{F}  |^{\aleph_0},$$ which implies $|\text{Aut}(A)| = |  \mathbb{F} |\,^{\aleph_0}.$

It remains to consider the case, when \ $\dim_{\mathbb{F}} eAe < \infty$ for all idempotents $e\in A.$ By Lemma \ref{L4_new!}, $A\cong M_{\infty}(\mathbb{F}).$ For an arbitrary mapping $f: \mathbb{N}\rightarrow \mathbb{F}$  consider the invertible infinite matrix $$ a_f = \text{Id} + \sum_{i=1}^{\infty} \ f(i) e_{2i-1, \, 2i},$$ where $\text{Id}$ is the identity $(\mathbb{N}\times \mathbb{N})$-matrix and $e_{i,j}$ are matrix units. The matrices $a_f$ are not finitary but $$ a_f^{-1} M_{\infty}(\mathbb{F}) a_f=M_{\infty}(\mathbb{F}) .$$ Let $\hat{a_f}$  denote the automorphism of conjugation by $a_f.$ The mapping $$ f \mapsto \hat{a_f} \in \text{Aut}(M_{\infty}(\mathbb{F}))$$ is injective since $$ a_f^{-1} e_{1,2i-1} a_f = e_{1,2i-1} + f(i) e_{1,2i} \quad \text{for} \quad i\geq 1. $$ Hence $$|\mathbb{F}  |^{\aleph_0} = |\text{Map}(\mathbb{N},\mathbb{F})| \leq |\text{Aut}(M_{\infty}(\mathbb{F}))| \leq |\text{Lin}_{\mathbb{F}}(A)|= |\mathbb{F}  |^{\aleph_0}.$$  \end{proof}


\begin{thebibliography}{99}

\bibitem{AyKud} S.~Ayupov and K.~Kudaybergenov, Infinite dimensional central simple regular algebras with outer derivations, {\em Lobachevskii Journal of Mathematics} {\bf 41}	(no.~3) (2020).

\bibitem{BezOl} O.~Bezushchak and  B.~Oliynyk, Unital locally matrix algebras  and Steinitz numbers,  {\em J. Algebra Appl.} (2020) doi: 10.1142/SO219498820501807.

\bibitem{BezOl_2} O.~Bezushchak and B.~Oliynyk, Primary decompositions of unital locally matrix algebras,  {\em Bull. Math. Sci.} (2020) doi: 10.1142/S166436072050006X.

\bibitem{Morita} O.~Bezushchak and B.~Oliynyk, Morita equivalent unital locally matrix algebras, {\em Algebra Discrete Math.}, {\bf 29} (2020), no.~2, pp. 173--179.

\bibitem{Glimm} J.~G.~Glimm,   On a certain class of operator algebras, {\em Trans. Amer. Math. Soc.} {\bf 95} (no.~2) (1960) 318--340.

\bibitem{Drozd_Kirichenko} Yu.~A.~Drozd, V.~V.~Kirichenko, Finite Dimensional Algebras, Springer-Verlag, Berlin--Heidelberg--New York (1994).

\bibitem{SkolemNoether} I.~N.~Herstein, Noncommutative Rings, Cambridge University Press (1968).

\bibitem{Jacobson} Jacobson~N. Lectures in abstract algebra. Volume 2.  Linear algebra, Springer-Verlag (1975). 

\bibitem{Koethe}  G.~K\"{o}the, Schiefk\"{o}rper unendlichen Ranges uber dem Zentrum, {\em  Math. Ann.} {\bf 105} (1931) 15--39.

\bibitem{Kurosh} A.~Kurosh, Direct decompositions of simple rings, 	{\em Rec. Math. [Mat. Sbornik] N.S.}  {\bf 11} ({\bf 53}) (no.~3) (1942) 245--264.

\bibitem{Pierce} R.~S.~Pierce, Associative Algebras, Graduate Texts in Mathematics, {\bf 88}, Springer, New York (1982).

\bibitem{Rubin}   Jean E.~Rubin, Set Theory for the Mathematician, New York: Holden-Day (1967). 

\bibitem{Strade} Strade, H.  Locally finite dimensional Lie algebras and their derivation algebras, {\em Abh. Math. Sem. Univ. Hamburg} {\bf 69} (1999) 373–-391.

\bibitem{Willard}   S.~Willard, General Topology, Mineola, New York: Dover Publications  (2004).

\end{thebibliography}
\end{document}